\documentclass[12pt]{amsart}
\usepackage{amscd}
\usepackage{amssymb}
\usepackage{a4wide}
\usepackage{amstext}
\usepackage{amsthm}
\usepackage{mathrsfs}
\usepackage[final]{hyperref}
\hypersetup{unicode= false, colorlinks=true, linkcolor=blue,
anchorcolor=blue, citecolor=green, filecolor=red, menucolor=blue,
pagecolor=blue, urlcolor=blue} 
 \DeclareMathOperator{\dist}{dist}

\DeclareMathOperator{\card}{card}
\DeclareMathOperator{\clos}{Clos}
\DeclareMathOperator{\lin}{Lin}

\renewcommand{\phi}{\varphi}

\newcommand{\pw}{\mathcal{P}W_\pi}

\newtheorem{Thm}{Theorem}%[section]
\newtheorem{theorem}[Thm]{Theorem}

\begin{document}
\sloppy
\title[A restricted shift completeness problem]{A restricted shift completeness problem}
\author{Anton Baranov, Yurii Belov, Alexander Borichev}
\address{Anton Baranov, 
\newline
Department of Mathematics and Mechanics,
St. Petersburg State University,\newline
St. Petersburg, Russia
\newline {\tt anton.d.baranov@gmail.com}
\newline\newline \phantom{x}\,\, Yurii Belov,
\newline 
Chebyshev Laboratory,
St. Petersburg State University,
St. Petersburg, Russia
\newline {\tt j\_b\_juri\_belov@mail.ru}
\newline\newline \phantom{x}\,\, Alexander Borichev, 
\newline Laboratoire d'Analyse, Topologie, Probabilit\'es, 
Aix--Marseille Universit\'e,\hfill\hfill\newline Marseille, France
\newline {\tt borichev@cmi.univ-mrs.fr}
}
\thanks{The first and the second authors were supported by the Chebyshev Laboratory 
(St. Petersburg State University) under RF Government grant 11.G34.31.0026.
The first author was partially supported by RFBR grant 11-01-00584-a 
and by Federal Program of Ministry of Education 2010-1.1-111-128-033.
The research of the third author was partially supported by the ANR grant FRAB}

\begin{abstract}
We solve a problem about the orthogonal complement  of the space spanned by restricted shifts of 
functions in $L^2[0,1]$
posed by M.Carlsson and C.Sundberg.
\end{abstract}

\maketitle

Recently, Marcus Carlsson and  Carl Sundberg posed the following problem. 
Let $f\in L^2[0,1]$. 
Consider the Fourier transform  
$$
\hat{f} (\lambda) = \int_0^1f(x)e^{i\lambda x}dx
$$
of $f$ and assume that the zeros of the entire function $\hat{f}$ 
are simple. Denote the set of these zeros by $\Lambda$. Suppose that ${\rm conv} ({\rm supp}\, f) = [0,1/2]$, 
and put 
$$
\mathfrak{A}_f = \clos_{L^2[0,1]}\lin\{\tau_t f: 0\leq t \leq 1/2 \},
$$
where $\tau_tf(x)=f(x-t)$. It is clear that 
$\{e^{i\lambda x}\}_{\lambda\in\Lambda}\perp \mathfrak{A}_f$ in $L^2[0,1]$. 
The problem by Carlsson and Sundberg is whether the family 
$$
\{e^{i\lambda x}\}_{\lambda\in\Lambda}\cup \{\tau_tf\}_{0\leq t\leq 1/2}
$$
is complete in $L^2[0,1]$. In this article we solve (a slightly more general form of) this problem. Our solution involves two components: 
a non-harmonic Fourier analysis in the Paley--Wiener space 
developed recently in \cite{bbb}, and sharp density results of 
Beurling--Malliavin type from \cite{pm,mp}.

\begin{theorem}\label{TA}
Let $0<a<1$, $f\in L^2[0,1]$, and let ${\rm conv} ({\rm supp}\, f) = [0,a]$. Denote by $\Lambda=\{(\lambda_k,n_k)\}$ the zero divisor of $\hat{f}$ {\rm(}i.e., $\hat{f}$ vanishes at $\lambda_k$ with 
multiplicity $n_k${\rm)}. Then the family 
$$
\{x^se^{i\lambda_k x}\}_{(\lambda_k,n_k)\in\Lambda,\,0\le s<n_k}\cup\{\tau_tf\}_{0\leq t\leq 1-a}
$$
is complete in $L^2[0,1]$.
\end{theorem}

However, in the limit case $a=1$ it is easy to show that the statement is not true.

\begin{theorem}\label{TB}
There exists $f\in L^2[0,1]$ such that 
${\rm conv} ({\rm supp}\, f) = [0,1]$, 
$\hat{f}$ has only simple zeros which form a set $\Lambda\subset\mathbb R$, 
and the family 
$$
\{e^{i\lambda x}\}_{\lambda\in \Lambda} \cup \{f\}
$$
is not complete in $L^2[0,1]$.
\end{theorem}

\begin{proof}[Proof of Theorem~\ref{TA}]
We apply the Fourier transform and a simple rescaling 
to reduce our problem to the following one. Let $F$ belong to 
the Paley--Wiener space $\mathcal{P}W_{\pi a}$ (the Fourier image 
of $L^2[-\pi a, \pi a]$), and let $\Lambda=\{(\lambda_k,n_k)\}$ be the zero divisor of $F$.
Then the family 
\begin{equation}
\label{prob}
\{F(z)e^{itz}\}_{|t|\leq \pi(1-a)}\cup \{K^s_\lambda\}_{(\lambda_k,n_k)\in\Lambda,\,0\le s<n_k}
\end{equation}
is complete in $\pw$. Here, 
$K^0_\lambda(z) = K_\lambda(z) = 
\frac{\sin [\pi (z-\overline \lambda)]}{\pi(z-\overline 
\lambda)}$ is the reproducing kernel of the space $\pw$, and 
$$
K^s_\lambda=\Bigl(\frac{d}{d\overline{\lambda}}\Bigr)^s K_\lambda
$$
reproduce the $s$-th derivatives:
$$ 
\langle f,K^s_\lambda \rangle_{\pw}=f^{(s)}(\lambda),\qquad 
f\in\pw,\,
\lambda\in\mathbb C,\, s\ge 0. 
$$
It is easy to show that for every $\beta\in\mathbb{R}$, the functions 
$$
F(z)\frac{\sin[\pi(1-a)(z-\beta)]}{z-\beta-2n(1-a)^{-1}}, 
\qquad n\in\mathbb{Z},
$$
belong to the closed linear span of 
$\{F(z)e^{itz}\}_{|t|\leq \pi(1-a)}$ in $\pw$. 
We set $G(z) = F(z)\sin[\pi(1-a)(z-\beta)]$, and 
fix $\beta$ in such a way that $G$ has only simple zeros. 
Denote $\Lambda' = \{\beta+\frac{2n}{1-a}\}_{n\in\mathbb{Z}}$.
It remains to verify that the family 
$$
\biggl{\{}\frac{G(z)}{z-\lambda}\biggr{\}}_{\lambda\in \Lambda'}
\cup\{K^s_{\lambda_k}\}_{(\lambda_k,n_k)\in\Lambda,\,0\le s<n_k}
$$ 
is complete in $\pw$. 

Assume the converse. Then there exists $h\in\pw\setminus\{0\}$ such that
\begin{equation}
\label{1perp}
\biggl{(}\frac{G(z)}{z-\lambda},h\biggr{)} = 0,\qquad \lambda\in \Lambda',
\end{equation}
\begin{equation}
\label{2perp}
(h,K^s_\lambda) = 0,\qquad (\lambda_k,n_k)\in\Lambda,\,0\le s<n_k.
\end{equation}

For $0\leq\gamma<1$, we expand $h$ with respect to the 
orthogonal basis $K_{n+\gamma}$:
$$
h=\sum_{n\in\mathbb{Z}}\bar{a}_{n,\gamma}K_{n+\gamma},\qquad \{a_{n,\gamma}\}\in \ell^2.
$$
Then \eqref{1perp}--\eqref{2perp} can be rewritten as 
\begin{align*}
\sum_{n\in\mathbb{Z}}\frac{a_{n,\gamma}G(n+\gamma)}
{n+\gamma-\lambda} &= 0,\qquad \lambda \in \Lambda',\\
\sum_{n\in\mathbb{Z}}\frac{\bar{a}_{n,\gamma}(-1)^n}
{(n+\gamma-\lambda_k)^s} &= 0, \qquad (\lambda_k,n_k)\in\Lambda,\,0<s\le n_k.
\end{align*}
Changing $\gamma$ if necessary 
we can assume that $a_{n,\gamma}\neq 0$, $G(n+\gamma)\not = 0$, 
$n\in\mathbb Z$.
Therefore there exist entire functions $S_\gamma$ and $T_\gamma$ such that 
\begin{align}
\label{1ef}
\sum_{n\in\mathbb{Z}}\frac{a_{n,\gamma}G(n+\gamma)}{n+\gamma - z} 
&= \frac{T_\gamma(z)\sin[\pi(1-a)(z-\beta)]}{\sin [\pi(z-\gamma)]},\\
\label{2ef}
\sum_{n\in\mathbb{Z}}\frac{\bar{a}_{n,\gamma}(-1)^n}{n+\gamma-z} 
&= \frac{S_\gamma(z)F(z)}{\sin [\pi(z-\gamma)]}=
\frac{h(z)}{\sin [\pi(z-\gamma)]}.
\end{align}
Since $h=FS_\gamma$ does not depend on $\gamma$, we write in what follows 
$S=S_\gamma$. 

Put $V_\gamma=S T_\gamma$. Comparing the residues 
in equations \eqref{1ef}--\eqref{2ef} at the points 
$n+\gamma$, $n\in\mathbb{Z}$, we conclude that
\begin{equation}
\label{altern}
V_\gamma(n+\gamma)=(-1)^n|a_{n,\gamma}|^2,\quad n\in\mathbb{Z}.
\end{equation}
By construction, $V_\gamma$ is of at most exponential type $\pi$. Therefore, we have the representation  
\begin{equation}
\label{Tfunc}
V_\gamma(z)= Q_\gamma(z) + \sin [\pi ( z - \gamma)] R_\gamma(z),
\end{equation}
where
$$
Q_\gamma(z) = \sin\pi (z - \gamma) \sum_{n\in\mathbb{Z}}
\frac{|a_{n,\gamma}|^2}{z - n - \gamma},
$$
and $R_\gamma$ is a function of zero exponential type. Thus, 
the conjugate indicator diagram of $V_\gamma$ is $[-\pi,\pi]$, and hence, the conjugate indicator diagram of 
$T_\gamma$ and $S$ are $[-\pi a, \pi a]$ and 
$[-\pi(1-a),\pi(1-a)]$ correspondingly. 
Therefore, each of the functions $V_\gamma^*/V_\gamma$, 
$T_\gamma^*/T_\gamma$, and $S^{*}/S$ is a ratio of two Blaschke 
products. Here we use the notation $H^*(z)=\overline{H(\bar{z})}$. 

It follows from \eqref{2ef} that
$$
\frac{S(z) F(z)}{\sin [\pi(z - \gamma)]}\cdot\frac{S^{*}(z)}{S(z)} 
= \sum_{n\in\mathbb{Z}}\frac{\bar{a}_{n,\gamma}(-1)^n}{n+\gamma-z}
\cdot\frac{S^{*}(n+\gamma)}{S(n +\gamma)} + H(z)
$$
for some entire function $H$. Since $FS^{*}\in\pw$, 
we conclude that $H$ is of zero exponential type and  
tends to $0$ along the imaginary axis. Thus, $H=0$.

We set $\bar{b}_{n,\gamma}=\bar{a}_{n,\gamma}\frac{S^{*}(n+\gamma)}{S(n +\gamma)}$, 
and obtain
$$
\sum_{n\in\mathbb{Z}}\frac{\bar{b}_{n,\gamma}(-1)^n}{n+\gamma-z} = 
\frac{S^{*}(z)F(z)}{\sin \pi(z-\gamma)}.
$$
Analogously, using the fact that the function $z \mapsto T_\gamma(z)\sin[\pi(1-a)(z-\beta)]$ 
belongs to $\pw$ and the fact that $ST_\gamma$ is real on $\mathbb{Z}+\gamma$, we deduce from \eqref{1ef} that 
$$
\sum_{n\in\mathbb{Z}}\frac{b_{n,\gamma}G(n+\gamma)}{n+\gamma-z} 
= \frac{T_\gamma^* (z)\sin[\pi(1-a)(z-\beta)]}{\sin [\pi(z-\gamma)]}.
$$
Thus, the function
$$
g=\sum_{n\in\mathbb{Z}}\bar{b}_n K_{n+\gamma}
$$ 
is orthogonal to the system \eqref{prob},
whence the elements $h+g$, $ih-ig$ are also orthogonal to \eqref{prob}, 
and correspond to the pairs $(S+S^{*}, T_\gamma+T_\gamma^*)$, 
$(iS-iS^{*}, -iT_\gamma+iT_\gamma^*)$. 
Therefore, from now on we assume that $S$, $T_\gamma$, and hence,  
$V_\gamma$ are real on the real line. 

Now it follows from (\ref{altern}) that the function 
$V_\gamma$ has at least one zero in every interval 
$(n+\gamma, n+1+\gamma)$, $n\in\mathbb{Z}$.
By \eqref{Tfunc}, the zeros of $V_\gamma$ coincide 
with the zeros of the function 
\begin{equation}
\label{type0}
R_\gamma(\lambda)+\sum_{n\in\mathbb{Z}}\frac{|a_{n,\gamma}|^2}{\lambda-n-\gamma}.
\end{equation} 

Next we fix $\gamma\in[0,1)$ and a sufficiently small $\delta>0$   
for which there exist two subsets $\Sigma, \Sigma_1$ of the zero set $\mathcal{Z}(S)$ of the function $S$ with the following properties:
\begin{itemize}
\begin{item}
$\Sigma$ has exactly one point in those intervals 
where $\mathcal{Z}(S)\cap[n+\gamma,n+1 +\gamma)\neq\emptyset$, 
and 
$$
\dist(x,\mathbb{Z}+\gamma)>\frac{\delta}{1+x^2}, \qquad x\in\Sigma;
$$
\end{item}
\begin{item}
$\Sigma_1$ has positive upper density, 
and $\dist(x, \mathbb{Z}+\gamma)>\delta$, $x\in\Sigma_1$.
\end{item}
\end{itemize}

From now on, we use the notations $R=R_\gamma$, 
$a_n=a_{n,\gamma}$, $V=V_\gamma$, $T=T_\gamma$. We need to consider three cases. 
If $R$ is a nonzero polynomial, then the zeros of the 
function \eqref{type0} approach $\mathbb{Z}+\gamma$ and we obtain 
a contradiction to the existence of $\Sigma_1$. 
If $R=0$, then \cite[Proposition 3.1]{bbb} implies that 
the density of $\Sigma_1$ is zero. Finally, if $R$ is not a polynomial, 
we can divide it by $(z-z_1)(z-z_2)$, where $z_1$ and $z_2$ 
are two arbitrary zeros of $R$, $z_1,z_2\not\in\Sigma$, 
to get a function $R_1$ of zero exponential type which is 
bounded on $\Sigma$. 

Next, we obtain some information on $\Sigma$.
For a discrete set $X=\{x_n\} \subset \mathbb{R}$ we 
consider its counting function  
$n_X(t) = {\rm card}\, \{n: x_n \in [0, t)\}$, $t\ge 0$,
and $n_X(t) = -{\rm card}\, \{n: x_n \in (-t, 0)\}$, $t<0$.
If $f$ is an entire function and $X$ is the set of its real zeros
(counted according to multiplicities), then there exists a branch
of the argument of $f$ on the real axis,
which is of the form $\arg f(t) = \pi n_X(t) +\psi(t)$, where $\psi$ 
is a smooth function. Such choice of the argument is unique
up to an additive constant and in what follows we always assume that the argument is chosen
to be of this form.

We use the (easy to show) fact that for every function $f\in\pw$ with the 
conjugate indicator diagram $[-\pi,\pi]$ and all zeros 
in $\overline{\mathbb{C}_+}$, one has
\begin{equation}
\label{argum}
\arg f = \pi x +\tilde{u}+c,
\end{equation}
where $u\in L^1((1+x^2)^{-1}dx)$, $c\in\mathbb{R}$.
Here $\tilde u$ denotes the conjugate function (the Hilbert
transform) of $u$,
$$
\tilde u (x)=\frac{1}{\pi} v.p. 
\int_\mathbb{R} \bigg(\frac{1}{x-t} +\frac{t}{t^2+1}\bigg) u(t)dt.
$$

It follows from \eqref{1ef}--\eqref{2ef} that $FV\in\mathcal{P}W_{\pi a+\pi}$.
Now let us replace all zeros 
$\lambda$ of the functions $h$, $F$, $S$, $T$, and $V$ in $\mathbb{C}_{-}$ 
by $\bar{\lambda}$. 
Since the Paley--Wiener space is closed under division by Blaschke products, we 
still have for the new functions $h$, $F$, $S$, $T$, and $V$  
(which we denote by the same letters) that 
$h\in \pw$ and $FV\in\mathcal{P}W_{\pi a + \pi}$. 
Recall that the function $V$ has at least one zero in each  of the intervals
$(n+\gamma, n+1+\gamma)$, $n \in \mathbb{Z}$.
Let us consider its representation $V=V_0 H$, where the zeros of $V_0$ 
are simple, interlacing with $\mathbb{Z}+\gamma$ 
and $V_0|_\Sigma=0$. It is clear that $\arg V_0 = \pi x+ O(1)$. 
Since, by (\ref{argum}), 
$$
\arg(FV)=  \pi a  x +\pi x + \tilde{u}+c,
$$
we conclude that
$$
\arg(FH)=\pi a x +\tilde{u}+O(1).
$$

Consider the equality $h=FHS/H$ and note that
$$
\arg\Big(\frac{S}{H}\Big)  = \pi n_\Sigma -\alpha,
$$
where $\alpha$ is some nondecreasing function on $\mathbb{R}$. 
This follows from the fact that
$S/H$ vanishes only on a subset of the real axis which contains $\Sigma$ 
and $\frac{S^*H}{SH^*}$ is a Blaschke product.
Applying the representation (\ref{argum}) to $h$, we conclude that 
\begin{equation}
\label{n15}
\pi n_\Sigma(x) = \pi(1-a)x + \tilde{u}+ v + \alpha,
\end{equation}
where $u\in L^1((1+x^2)^{-1}dx)$, $v\in L^\infty(\mathbb{R})$, 
and $\alpha$ is nondecreasing.

Summing up, we have an entire function $R_1$ of zero exponential type which is not a polynomial, and which is bounded 
on a set $\Sigma\subset\mathbb R$ satisfying \eqref{n15}.

To deduce a contradiction from this, we use some information on the classical Polya problem
and on the second Beurling--Malliavin theorem.
We say that a sequence $X=\{x_n\} \subset \mathbb{R}$ is a {\it Polya sequence} 
if any entire function of zero exponential type which is bounded on $X$
is a constant. We say that a disjoint sequence of intervals 
$\{I_n\}$ on the real line is \itshape a long sequence of intervals 
\normalfont if 
$$
\sum_n\frac{|I_n|^2}{1+\dist^2(0,I_n)}=+\infty.
$$

A complete solution of the Polya problem was obtained by 
Mishko Mitkovski and Alexei Poltoratski \cite{mp}. In particular\footnote{
The ``only if'' part of this statement is implicitly contained in the results of Louis de Branges in the 1960-s:
\cite[Theorem XI]{deb1}, \cite[Theorems 66, 67]{deb2}; see also \cite[Remark, page 1068]{mp}.
}, a separated sequence 
$X\subset\mathbb{R}$ is not a Polya sequence if and only 
if there exists a long sequence of intervals $\{I_n\}$ such that 
$$
\frac{\card(X \cap I_n)}{|I_n|}\rightarrow 0.
$$

Applying this result to our $R$ and $\Sigma$ (formally speaking, $\Sigma$ is not a separated sequence but by construction it is a union of two separated
sequences which are interlacing), we find a long system 
of intervals $\{I_n\}$ such that 
$$%\begin{equation}\label{long}
\frac{\card(\Sigma\cap I_n)}{|I_n|}\rightarrow0.
$$%\end{equation}

Given $I=[a,b]$, denote $I^-=[a,(2a+b)/3]$, $I^+=[(a+2b)/3,b]$, 
$$
\Delta^*_I = \inf_{I^+}[\pi(1-a)x -\pi n_\Sigma(x) +v] -
\sup_{I^-}[\pi(1-a)x - \pi n_\Sigma(x) +v].
$$
Now, for a long system of intervals $\{I_n\}$ and for some $c>0$ 
we have 
$$
\Delta^*_{I_n}\geq c|I_n|.
$$

Next we use a version of the second Beurling--Malliavin theorem 
given by Nikolai Makarov and Alexei Poltoratskii in \cite[Theorem 5.9]{pm}. 
This theorem (or rather its proof)  
gives that 
if the function $\pi(1-a)x - \pi n_\Sigma(x) +v$ may be represented as 
$-\alpha -\tilde u$ for $\alpha$ and $u$ as above, then 
there is no such long family of intervals. This contradiction completes the proof.
\end{proof}

\begin{proof}[Proof of Theorem~\ref{TB}]
Let $\Lambda_1=\{\pi 2^n\}_{n\geq1}$, 
$\Lambda = 2\pi \mathbb{Z}\setminus\Lambda_1$. Consider the entire 
function $F$ of zero order with zero set $\Lambda_1$, and denote $G(z)=e^{iz\slash2}\cdot\frac{\sin(z\slash2)}{F(z)}$. Then $G=\hat{f}$ for some $f\in L^2[0,1]$, and $L^2[0,1]\ominus\{e^{i\lambda x}\}_{\lambda\in\Lambda}$ is of infinite dimension.
\end{proof}

{\bf Acknowledgements.} The authors are thankful to Misha Sodin 
for helpful discussions. A part of the present work was done when the authors participated 
in the workshop "Operator Related Function Theory" organized by Alexandru Aleman and Kristian Seip at the Erwin Schr\"odinger International Institute for Mathematical Physics (ESI). The hospitality of ESI is greatly appreciated.

\end{document}